\documentclass[12pt]{amsart}
\usepackage{amsmath,amssymb}

\theoremstyle{plain}
\newtheorem{thm}{Theorem}[section]
\newtheorem{theorem}[thm]{Theorem}

\newtheorem{lemma}[thm]{Lemma}
\newtheorem{corollary}[thm]{Corollary}
\newtheorem{proposition}[thm]{Proposition}
\theoremstyle{definition}

\newtheorem{remark}[thm]{Remark}

\numberwithin{equation}{section}

\newcommand{\sC}{{\mathcal C}}

\newcommand{\sK}{{\mathcal K}}
\newcommand{\sL}{\mathcal{L}}

\newcommand{\sN}{{\mathcal N}}
\newcommand{\sO}{{\mathcal O}}

\newcommand{\sU}{{\mathcal U}}

\newcommand{\A}{{\mathbb A}}

\newcommand{\C}{{\mathbb C}}

\newcommand{\bG}{{\mathbb G}}

\newcommand{\bP}{{\mathbb P}}

\newcommand{\id}{{\rm id}}

\newcommand{\diag}{{\rm diag}}
\newcommand{\n}{{\rm n}}

\newcommand{\PSp}{{\rm PSp}}

\newcommand{\SO}{{\rm SO}}

\newcommand{\fsl}{{\mathfrak s}{\mathfrak l}}
\newcommand{\fgl}{{\mathfrak g}{\mathfrak l}}
\newcommand{\fsp}{{\mathfrak s}{\mathfrak p}}
\newcommand{\fcsp}{{\mathfrak c}{\mathfrak s}{\mathfrak p}}

\newcommand{\fg}{{\mathfrak g}}
\newcommand{\fh}{{\mathfrak h}}
\newcommand{\fl}{{\mathfrak l}}

\newcommand{\fp}{{\mathfrak p}}

\newcommand{\aut}{{\mathfrak a}{\mathfrak u}{\mathfrak t}}

\newcommand{\BP}{{\mathbb P}}
\newcommand{\0}{{\mathcal O}}

\def\RatCurves{\mathop{\rm RatCurves}\nolimits}
\def\Aut{\mathop{\rm Aut}\nolimits}

\def\GL{\mathop{\rm GL}\nolimits}

\def\Sym{\mathop{\rm Sym}\nolimits}
\def\Pic{\mathop{\rm Pic}\nolimits}
\def\End{\mathop{\rm End}\nolimits}
\def\Hom{\mathop{\rm Hom}\nolimits}

\def\PGL{\mathop{\rm PGL}\nolimits}
\def\SL{\mathop{\rm SL}\nolimits}

\def\min{\mathop{\rm min}\nolimits}

\author{Michel Brion and Baohua Fu}

\title{Minimal rational curves on wonderful group compactifications}

\begin{document}

\begin{abstract}
Consider a simple algebraic group $G$ of adjoint type, and its wonderful 
compactification $X$. We show that $X$ admits a unique family of minimal 
rational curves, and we explicitly describe the subfamily consisting of 
curves through a general point. As an application, we show that $X$ has
the target rigidity property when $G$ is not of type $A_1$ or $C$. 
\end{abstract}

\maketitle

\section{Introduction}

Throughout this note, we consider algebraic varieties 
over the field of complex numbers. For a uniruled projective 
manifold $X$, let $\RatCurves^\n(X)$ 
denote the normalization of the space of rational curves on $X$ 
(see \cite[II.2.11]{Kollar}). Every irreducible component $\sK$ of 
$\RatCurves^\n(X)$ is a (normal) quasi-projective variety equipped with
a quasi-finite morphism to the Chow variety of $X$; the image consists
of the Chow points of irreducible, generically reduced rational curves.  
Also, there is a universal family $\sU$ with projections 
$\rho : \sU \to \sK$, $\mu : \sU \to X$, and $\rho$ is 
a $\bP^1$-bundle (for these results, see 
\cite[II.2.11, II.2.15]{Kollar}).

For any $x \in X$, let $\sU_x := \mu^{-1}(x)$ and $\sK_x := \rho(\sU_x)$; 
then $\sK$ is called a \emph{family of minimal rational curves} 
if $\sK_x$ is non-empty and projective for a general point $x$. 
There is a rational map $\tau: \sK_x \dasharrow \bP T_x(X)$ 
(the projective space of lines in the tangent space at $x$) 
that sends any curve which is smooth at $x$ to its tangent direction. 
The closure of the image of $\tau$ is denoted by $\sC_x$ 
and called the \emph{variety of minimal rational tangents} 
(VMRT) at the point $x$. 
By \cite[Thm.~1]{HM2} and \cite[Thm.~3.4]{Kebekus}, 
composing $\tau$ with the normalization map $\sK_x^\n \to \sK_x$ 
yields the normalization of $\sC_x$. Also, $\sK_x^\n$ is
a union of components of the variety $\RatCurves^\n(x, X)$ defined
in \cite[II.2.11.2]{Kollar}, and hence is smooth for $x \in X$ 
very general by \cite[II.3.11.5]{Kollar}.

The projective geometry of $\sC_x \subset \bP T_x(X)$ encodes many 
geometric properties of $X$, and may be used effectively for solving
a number of problems on uniruled varieties (see e.g. \cite{Hwang}). 
When $X$ has Picard number one, the VMRT has been determined in many 
cases, including the homogeneous manifolds $G/P$, where $G$ is
a simple linear algebraic group and $P$ a maximal parabolic subgroup
(see \cite[Prop.~1]{HM1} for the case where $P$ corresponds to
a long simple root, and \cite[Prop.~3.5]{HM3}, 
\cite[Prop.~3.2.1, Prop.~8.1.3]{HM4} for the remaining cases;
see also \cite[Thm.~4.8]{LM}).  
But very few examples with large Picard number are known. 
The case of complete toric manifolds is worked out in \cite{CFH}, 
where it is shown that such a manifold may admit several families 
of minimal rational curves.

The aim of this note is to study minimal rational curves on wonderful 
compactifications of semisimple algebraic groups of adjoint type. 
For any such group $G$, De Concini and Procesi introduced in \cite{dCP} 
its wonderful compactification $X$. This is a projective manifold 
equipped with an action of $G \times G$ and containing $G$ as 
an open orbit (where $G \times G$ acts on $G$ by left and right 
multiplication). Moreover, the boundary $X \setminus G$
is a union of $\ell$ irreducible divisors $D_1, \cdots, D_\ell$ 
with smooth normal crossings, where $\ell$ is the rank of $G$, 
and the $G \times G$-orbit closures in $X$ are exactly the partial
intersections $D_{i_1} \cap \cdots \cap D_{i_s}$. 

The manifold $X$ is rational, since so is $G$; moreover, 
$X$ is Fano with Picard number $\ell$. Also, by homogeneity, 
the VMRT at every point $x$ of the open orbit $G$ is well-defined 
and independent of $x$. We may thus choose for $x$ the neutral 
element of $G$, called the base point of $X$; then $T_x(X)$ is 
the Lie algebra $\fg$ of $G$. The isotropy group of $x$ in
$G \times G$ is $G$ embedded diagonally; it acts on $T_x(X)$
via the adjoint representation, and the induced action on
$\bP \fg$ stabilizes the VMRT $\sC_x$.

We may now state the main result of this note:

\begin{theorem}\label{thm:main}
Let $X$ be the wonderful compactification of a simple algebraic 
group $G$ of adjoint type. Then 

\item{\rm (i)} There exists a unique family of minimal rational curves
$\sK$ on $X$. Moreover, $\sK_x$ is smooth and the rational
map $\tau : \sK_x \dasharrow \sC_x$ is an isomorphism.

\item{\rm (ii)} $\sC_x$ is the closed $G$-orbit in $\bP \fg$, 
if $G$ is not of type $A$.

\item{\rm (iii)} When $G$ is of type $A_\ell$ ($\ell \geq 2$), 
so that $G = \PGL(V)$ for a vector space $V$ of dimension $\ell + 1$, 
the VMRT $\sC_x$ is the image of $\bP V \times \bP V^*$ 
under the Segre embedding $\bP V \times \bP V^* \to \bP \End(V)$, 
followed by the projection 
$\bP \End(V) \dasharrow \bP(\End(V)/ \C \id) = \bP \fg$ with center 
$x = [\id]$.
\end{theorem}

The uniqueness of a family of minimal rational curves on $X$ 
is somewhat surprising, since $X$ can have an arbitrary large Picard 
number. The key point is the uniqueness of a $B$-stable irreducible curve
through $x$, where $B$ is a Borel subgroup of $G$ acting on $X$ by
conjugation (see Proposition \ref{prop:irr} for a more general result). 
When $G$ is not of type $A$, it turns out that the VMRT has the same 
dimension as the closed orbit in $\bP \fg$, which implies claim (ii). 
When $G = \PGL(V)$, we construct additional minimal rational curves 
through $x$ by using multiplicative one-parameter subgroups. 
In contrast with the rational homogeneous manifolds, the wonderful 
compactification $X$ is not covered by lines when $\ell \geq 2$ 
(see Remark \ref{rem:lines}; when $\ell = 1$, we have $G = \PGL_2$ 
and $X \cong \bP^3$, hence $\sK_x \cong \bP^2$ and $\sC_x \cong \bP \fg$). 

In the general case, where $G$ is no longer assumed simple,
we have a unique decomposition into simple factors,
\begin{equation}\label{eqn:G} 
G = G_1 \times \cdots \times G_m. 
\end{equation}
Then accordingly,
\begin{equation}\label{eqn:X} 
X = X_1 \times \cdots \times X_m, 
\end{equation} 
where $X_i$ denotes the wonderful compactification of $G_i$.

\begin{proposition}\label{prop:prod}
With the above notation and assumptions, let $\sK$ be a family 
of minimal rational curves on $X$. Then $\sK$ consists of curves 
on a unique factor $X_i$. In particular, $\sC_x$ 
is the VMRT of $X_i$ at its base point.
\end{proposition}

As an application, we shall show that the wonderful group 
compactification $X$ has the Liouville property if $G$ is simple 
and not of type $A_1$ or $C$. This implies the target rigidity 
property: for any such $X$ and any projective variety $Y$, every 
deformation of a surjective morphism $f: Y \to X$ comes from 
automorphisms of $X$.

This note is organized as follows. In Section \ref{sec:aux},
we collect auxiliary results about rational curves on almost 
homogeneous manifolds. Some special rational curves are investigated
in Section \ref{sec:1ps}, namely, closures of multiplicative or additive 
one-parameter subgroups in the wonderful compactification. 
Section \ref{sec:proofs} contains the proofs of Theorem \ref{thm:main} 
and Proposition \ref{prop:prod}. Some geometric constructions of 
the minimal rational curves are presented in Section \ref{sec:geom}, 
and the above application in the final Section \ref{sec:appl}. 

\medskip

{\em Acknowledgments:} We are very grateful to Jun-Muk Hwang, 
who initiated this problem and also conceived the application. 
We would like to thank St\'ephane Druel, Nicolas Perrin, Colleen Robles
and Nanhua Xi for helpful discussions. Baohua Fu is supported 
by National Scientific Foundation of China (11225106 and 11321101).

\section{Some auxiliary results}
\label{sec:aux}

Throughout this section, $X$ denotes a projective manifold
on which a connected linear algebraic group $G$ acts with an 
open orbit $X^0$, that is, $X$ is almost homogeneous under $G$. 
We choose a base point $x \in X^0$ and denote by 
$H \subset G$ its isotropy group; the corresponding
Lie algebras will be denoted by $\fh \subset \fg$. 
Thus, the orbit $X^0 = G \cdot x$ is identified with the 
homogeneous space $G/H$. Moreover, $H$ acts on 
the tangent space $T_x(X)$ via the isotropy representation, 
identified to the quotient representation $\fp := \fg/\fh$.

We begin with the following observation:

\begin{lemma}\label{lem:gen}
\item{\rm (i)} Let $f: \bP^1 \to X$ be a morphism which is birational 
over its image. If this image meets $X^0$, then $f^*(T_X)$ is generated
by its global sections.

\item{\rm (ii)} Let $\sK$ be a covering family of rational curves on $X$.
Then $\sK_x$ is smooth; moreover, $H$ acts on $\sK_x$ and 
permutes transitively its components.
\end{lemma}

\begin{proof}
(i) follows from \cite[II.3.11]{Kollar}. We give a direct proof in this particular case.
  The $G$-action on $X$ yields a map of sheaves 
$\sO_X \otimes \fg \to T_X$, which is surjective over $X^0$. 
So we obtain a generically surjective map 
$\sO_{\bP^1} \otimes \fg \to f^*(T_X)$. This yields the
assertion by using the fact that every vector bundle on $\bP^1$
is a direct sum of line bundles $\sO_{\bP^1}(n)$, and the
vanishing of $H^0(\bP^1, \sO_{\bP^1}(n))$ for $n < 0$.

(ii) Let $f$ as in (i). Then $f$ is free in the sense of 
\cite[II.3.1]{Kollar}; it follows that $\Hom(\bP^1,X)$ is smooth 
at $f$, by \cite[II.3.5.4]{Kollar}. Thus, $\RatCurves^\n(X)$ 
is smooth at its point corresponding to $f$, in view of 
\cite[II.2.15]{Kollar}. As a consequence, $\sK$ is smooth
at any curve meeting $X^0$.

Consider the universal family $\sU$ with maps $\rho : \sU \to \sK$,
$\mu : \sU \to X$ as in the introduction. Then $G$ acts on $\sU$
and $\sK$; moreover, $\rho$ and $\mu$ are equivariant. Thus, 
$\sU^0 := \mu^{-1}(X^0)$ (resp. $\sK^0 := \rho(\sU^0)$) 
is an open $G$-stable subset of $\sU$ (resp. $\sK$). 
Moreover, the restriction 
$\rho^0 : \sU^0 \to \sK^0$ is a $\bP^1$-bundle. Since $\sK^0$ is 
smooth, so is $\sU^0$. But $\sU^0$ is the associated fiber bundle
$G \times^H \sU_x$ over $X^0 = G/H$. Thus, $\sU_x$ is smooth,
and hence so is $\sK_x$ since $\sU_x \to \sK_x$ is a 
$\BP^1$-bundle. Also, $\sU^0$ is irreducible; thus,
$H$ acts transitively on the set of components of
$\sU_x$, or equivalently of $\sK_x$.
\end{proof}

Next, we obtain a key technical result:

\begin{lemma}\label{lem:curves}
Assume that $H$ is reductive and the isotropy representation 
$\fp$ is a multiplicity-free sum of irreducible representations 
of $\fh$ with linearly independent highest weights 
(e.g., $\fp$ is irreducible as a representation of $\fh$).
Let $C \subset X$ be an irreducible curve through $x$,
stable under a Borel subgroup $B \subset H$, and set 
$C^0 := C \cap X^0$. Then

\item{\rm (i)} $C^0$ is smooth and $B$-equivariantly isomorphic 
to its tangent line $L$ at $x$, which is a highest weight line
in $\fp$. Moreover, $L$ determines $C$ uniquely, and the stabilizer 
of $C$ in the neutral component $H^0$ equals the stabilizer of $L$. 

\item{\rm (ii)} There exists an additive one-parameter subgroup 
$u : \bG_a \to G$ such that $C = \overline{u(\bG_a) \cdot x}$.
\end{lemma}

\begin{proof}
(i) We first claim that $C$ is not fixed pointwise by $B$. Indeed,
the fixed point locus $(G/H)^B$ equals $(G/H)^{H^0}$,
since $H^0/B$ is complete and $G/H$ is affine. Moreover,
since $H^0$ is reductive, $(G/H)^{H^0}$ is smooth and 
$T_x(G/H)^{H^0} = (\fg/\fh)^{H^0} = \fp^{\fh}$. The latter vanishes,
since the highest weights of $\fh$ in $\fp$ are all nonzero. 
Thus, $x$ is an isolated fixed point of $(G/H)^B$; this implies 
our claim.

By that claim, $B$ has an open orbit in $C^0$, say $B \cdot y$.
Thus, the isotropy subgroup $B_y \subset B$ has codimension $1$.

Next, we choose a maximal torus $T \subset B$; then
$B = T U$, where $U \subset B$ denotes the unipotent part.
We claim that $C$ is not fixed pointwise by $T$.
Otherwise, $T \subset B_y$ and hence $B_y = T U_y$, 
where $U_y \subset U$ is a subgroup of codimension $1$, normalized
by $T$. Moreover, $T$ acts trivially on $B/B_y \cong U/U_y$,
and hence on its tangent space at the base point.
It follows that $T$ has a non-zero fixed point in a quotient
of the Lie algebra of $U$. But this is impossible, since $H$ is
reductive.

Therefore, we may assume that $T \cdot y$ is open in $C^0$. 
Since $x \in C^0$ is $T$-fixed and $C^0$ is affine, it follows 
that $x \in \overline{T \cdot y}$;
in particular, $x \in \overline{H^0 \cdot y}$. 
So $C^0$ is contained in the fiber at $x$ of the 
geometric invariant theory quotient 
$G/H \to H^0 \backslash \! \backslash G/H$. 
By \cite[II.1, III.1]{Luna}, this fiber is 
$H^0$-equivariantly isomorphic to the nilcone $\sN$ of $\fp$ 
(the fiber at $0$ of the quotient $\fp \to \fp / \! / H^0$). 
Thus, we may identify $C^0$ with a $B$-stable curve $L$ 
in $\sN$. Moreover, $C$ and $L$ have the same stabilizer
in $H$.

If $U$ fixes $C^0$ pointwise, then $C^0$ is contained in $\fp^U$.
The latter is a multiplicity-free representation of $T$
with linearly independent weights. It follows that $C^0$
is a highest weight line in $\fp$.

We may thus assume that $U$ acts nontrivially on $C$;
equivalently, $B_y$ does not contain $U$. But $B_y$ has codimension 
$1$ in $B$, and hence contains a maximal
torus of $B$; so we may assume that $y$ is fixed by $T$.
Then $x \notin \overline{T \cdot y}$, a contradiction.

(ii) Denote by $\lambda$ the weight of $B$ in $L$; then
$L$ is the weight space $\fp^U_\lambda$. Since $H$ is reductive,
the projection $\fg \to \fg/\fh = \fp$ induces a surjective
map $\fg^U_\lambda \to \fp^U_\lambda$. So we may choose
$\xi \in \fg^U_\lambda$ whose projection spans $L$. Since 
$\lambda \neq 0$, we see that $\xi$ is nilpotent, and hence
defines an additive one-parameter subgroup $u : \bG_a \to G$
such that $b \, u(t) \, b^{-1}  = u(\lambda(b)t)$ for all $b \in B$ and 
$t \in \bG_a$. Thus, $\overline{u(\bG_a) \cdot x}$ is 
an irreducible $B$-stable curve with tangent line $L$ at $x$.
By uniqueness, this curve equals $C$. 
\end{proof}

\begin{remark}\label{rem: sym}
One may check that the assumptions of Lemma \ref{lem:curves} 
hold for any symmetric space $G/H$ of adjoint type, i.e., 
$G$ is semisimple of adjoint type and 
$H$ is the fixed point subgroup of an involutive automorphism of 
$G$. But these assumptions are not satisfied when $X$ is a toric 
variety (then $G$ is a torus and $H$ is trivial).
\end{remark}

\begin{lemma}\label{lem:smooth}
Consider a representation $V$ of the additive group $\bG_a$,
and a point $x \in \bP(V)$ which is not fixed by $\bG_a$.
Then the orbit map $\bG_a \to X$, $t \mapsto t \cdot x$ 
extends to an isomorphism 
$f: \bP^1 \to \overline{\bG_a \cdot x}$.
\end{lemma}

\begin{proof}
Clearly, $f$ is bijective, and smooth at every $t \in \bG_a$.
So it suffices to show that $f$ is smooth at $\infty$ as well.

Choose a representative $v \in V$ of $x$. We may assume that
$V$ is spanned by the orbit $\bG_a \cdot v$. Also, $V$ is a
direct sum of indecomposable representations of $\bG_a$,
and these are of the form $\C[x,y]_n$ (the space of homogeneous
polynomials of degree $n$ in $x$,  $y$), where $\bG_a$ acts via
$(t \cdot f)(x,y) := f(x, tx + y)$. Using a suitable projection,
we may assume that $V = V_n$; then $v = a_0 x^n + \cdots + a_n y^n$,
where $a_n \neq 0$. So 
$t \cdot v = (a_n x^n) t^n + (na_n x^{n-1} y + a_{n-1} x^n) t^{n-1} + \ldots$
In particular, the coefficients of $t^n$ and $t^{n-1}$ in 
$t \cdot v$ are both non-zero; this yields our assertion.
\end{proof}

\begin{lemma}\label{lem:fix}
Let $Y$ and $Z$ be projective varieties equipped with an action 
of $G$, and let $f: Y \to Z$ be a finite $G$-equivariant morphism. 
Assume that $Z$ contains a unique closed $G$-orbit $Z_{\min}$
and is equivariantly isomorphic to a subvariety of $\bP V$ 
for some finite-dimensional representation $V$ of $G$. Then 
$Y$ contains a unique closed $G$-orbit $Y_{\min}$ as well.
Moreover, $Y_{\min}$ is isomorphic to $Z_{\min}$ via $f$.
\end{lemma}

\begin{proof}
Choose a Borel subgroup $B \subset G$. Then the fixed point locus 
$Z^B \subset Z$ consists of a unique point, $z \in Z_{\min}$.
Indeed, for any $x \in Z^B$, the stabilizer $G_x$ is a parabolic
subgroup of $G$, hence the orbit $G \cdot x$ is closed and
$(G \cdot x)^B$ consists of the point $x$.

Denote by $U \subset B$ the unipotent part, and consider the 
fixed point locus $Z^U$. Then $Z^U$ is connected by 
\cite[Thm.~6.2]{Horrocks}; it is also $B$-stable, and $B$ 
acts on $Z^U$ via its quotient torus $T = B/U$. 
Moreover, $(Z^U)^T = Z^B$ consists of $z$ only.
As $Z^U$ is $T$-equivariantly isomorphic to a subvariety of 
$\bP V$, it follows that $Z^U$ is just the point $z$ (indeed, 
for any $x \in \bP V$ which is not fixed by $T$, the orbit closure 
$\overline{T \cdot x}$ contains at least $2$ fixed points of $T$). 

Next, consider $Y^U$: it is also connected, 
and finite since so is $f$. Thus, $Y^U$ is a unique point, say $y$.
So $Y^B$ is the point $y$ as well, and $Y$ contains a unique
closed $G$-orbit, namely, $Y_{\min} = G \cdot y$. Moreover, $f$
restricts to a finite equivariant map $Y_{\min} \to Z_{\min} = G \cdot z$.
Since $G_z$ (a parabolic subgroup of $G$) is connected, this implies 
the assertion.
\end{proof}

We now come to the main result of this section:

\begin{proposition}\label{prop:irr}
Assume that $H$ is reductive, and the representation of 
$\fh$ in $\fp$ is irreducible and non-trivial. Then

\item{\rm (i)} There exists a unique family of minimal rational curves
$\sK$ on $X$.

\item{\rm (ii)} $\sK_x$ contains a unique closed orbit of $H^0$; 
it is isomorphic to the orbit of the highest weight line in $\fp$. 

\item{\rm (iii)} $\sK_x$ is smooth and connected.

\item{\rm (iv)} $\mu \times \rho : \sU \to X \times \sK$ is a
closed immersion (i.e., $\sK$ consists of embedded curves). 
\end{proposition}

\begin{proof}
(i) and (ii) Let $\sK \subset \RatCurves^\n(X)$ be a family 
of minimal rational curves. Consider the action of a Borel subgroup 
$B \subset H$ on the projective variety $\sK_x$. By construction,
$\sK_x$ has an equivariant finite morphism to the Chow variety 
of $X$. In view of Lemma \ref{lem:curves} and Borel's fixed point 
theorem, the image of $\sK_x$ contains a unique closed $H^0$-orbit: 
that of the Chow point of the irreducible $B$-stable
curve $C$ that corresponds to the highest weight line $L$ in $\fp$. 
Applying Lemma \ref{lem:fix}, it follows that $\sK_x$ also 
contains a unique closed $H^0$-orbit, isomorphic to that of $L$. 
This implies that $\sK$ is the family of deformations of $C$, 
and hence is unique. 

(iii) The smoothness of $\sK_x$ follows from Lemma \ref{lem:gen},
and the connectedness from the fact that $\sK_x$ contains a unique
$B$-fixed point.

(iv) Let $p_2 : X \times \sK \to \sK$ denote the projection.
Then $p_2 \circ (\mu \times \rho) = \rho$ is proper and flat.
Applying \cite[I.1.10.1]{Kollar} together with Borel's fixed point 
theorem again, it suffices to show that the restriction of $\mu$
to $\rho^{-1}(C)$ is a closed immersion, where $C$ denotes, as 
above, the irreducible $B$-stable curve through $x$. But this follows
from \cite[II.3.5.4]{Kollar} together with Lemmas \ref{lem:curves}
and \ref{lem:smooth}, since $C$ is the image of a closed immersion
$f : \bP^1 \to X$ which is a free morphism.
\end{proof}

For later use, we also record the following easy result: 

\begin{lemma}\label{lem:comp}
Assume that $H$ is reductive and the boundary 
$\partial X := X \setminus X^0$ is a unique $G$-orbit. 
Let $Y$ be another $G$-equivariant compactification of $G/H$,
where $Y$ is complete (possibly singular). Then there is
a bijective morphism $f : X \to Y$ which restricts to the
identity on $G/H$. 
\end{lemma}

\begin{proof}
Note that $G/H$ is affine; thus, $\partial X$ and $\partial Y$ 
have pure codimension $1$. Since $X$ is smooth and $Y$ is
complete, the identity map of $G/H$, viewed as a rational map
$f : X \dasharrow Y$, is defined in codimension $1$. As 
$\partial X$ is homogeneous, it follows that $f$ is a morphism.
Thus, $\partial Y$ is a unique $G$-orbit and $f$ restricts 
to a finite equivariant map $g : \partial X \to \partial Y$, 
hence $g$ is an isomorphism. This implies that $f$ is bijective.
\end{proof}

\section{Closures of one-parameter subgroups}
\label{sec:1ps}

From now on, $G$ denotes a semisimple linear algebraic group 
of adjoint type and rank $\ell \geq 2$, and $\fg$ the 
corresponding (semisimple) Lie algebra.

We first introduce some notation: we choose a Borel subgroup 
$B \subset G$ as well as a maximal torus $T \subset B$. 
Then $B = T U$, where $U \subset B$ denotes
the unipotent part. We denote by $R$ the root system of $(G,T)$ and 
by $R^+ \subset R$ the subset of positive roots consisting of 
roots of $B$. The corresponding set of simple roots is denoted 
by $S = \{ \alpha_1, \cdots, \alpha_\ell \}$; we use the ordering 
of simple roots as in \cite{Bourbaki}. The half-sum of positive
roots is denoted by $\rho$. Also, we denote by $W$ the Weyl group, 
and by $w_0 \in W$ the unique element that sends $R^+$ to $-R^+$; 
then $w_0^2 = 1$ and $-w_0$ stabilizes $S$.

The coroot of any $\alpha \in R$ is denoted by $\alpha^\vee$;
this is a one-parameter subgroup of $T$. 
The coroots form the dual root system $R^\vee$. The pairing 
between characters and one-parameter subgroups is denoted by 
$\langle, \rangle$; we have 
$\langle \alpha, \alpha^\vee \rangle = 2$ for any $\alpha \in R$.

We denote by $\Lambda$ the weight lattice, with the submonoid 
$\Lambda^+$ of dominant weights, and the fundamental weights 
$\varpi_1,\ldots,\varpi_\ell$. For any $\lambda \in \Lambda^+$, 
we denote by $V(\lambda)$ the irreducible representation of 
the simply-connected cover of $G$ having highest weight 
$\lambda$. We then have a projective representation 
\[ \varphi_\lambda : G \to \PGL(V(\lambda)). \] 
Moreover, the $G$-orbit of the highest weight line in $V(\lambda)$
yields the unique closed orbit in the projectivization $\bP V(\lambda)$; 
it is isomorphic to $G/P_\lambda$, where the parabolic subgroup 
$P_\lambda$ only depends on the type of $\lambda$, i..e., 
the set of simple roots that are orthogonal to that weight. 

The closure of the image of $\varphi_\lambda$ in 
the projective space $\bP \End V(\lambda)$ will be denoted
by $X_\lambda$. This is a projective variety on which 
$G \times G$ acts via its action on $\bP \End V(\lambda)$ by
left and right multplication. Moreover, $X_\lambda$ contains
a unique closed orbit of $G \times G$; it is isomorphic to
$G/Q_\lambda \times G/P_\lambda$, where $Q_\lambda$ denotes
the parabolic subgroup containing $T$ and opposite to $P_\lambda$. 

When the dominant weight $\lambda$ is regular, $X_\lambda$ 
turns out to be smooth and independent of the choice of 
$\lambda$; this defines the wonderful compactification 
$X$ of $G$. The boundary $X \setminus G$ is a union of $\ell$ 
irreducible divisors $D_1, \cdots, D_\ell$ with smooth normal 
crossings. The $G \times G$-orbits in $X$ are 
indexed by the subsets of $\{1,\ldots, \ell \}$, by assigning 
to each such subset $I = \{ i_1, \ldots, i_s \}$ 
the unique open orbit $\sO_I$ in the partial intersection 
$D_{i_1} \cap \cdots \cap D_{i_s}$. In particular, 
the open orbit $X^0$ is $\sO_{\emptyset} = (G \times G)/ \diag(G)$,
and the closed orbit is $\sO_{\{1,\cdots, \ell\}}$.
Each orbit closure is equipped with a $G \times G$-equivariant 
fibration 
\[ f_I : \overline{\sO_I} \to G/Q_I \times G/P_I, \]
where $P_I$ denotes the parabolic subgroup associated
with the dominant weight $\sum_{i \in I} \varpi_i$, 
and $Q_I$ stands for the opposite parabolic subgroup. 
The fiber of $f_I$ at the base point of $G/Q_I \times G/P_I$
is isomorphic to the wonderful compactification of the
adjoint group of $L_I := P_I \cap Q_I$ (a Levi subgroup of both).
In particular, the closed orbit $\sO_{1,\ldots,\ell}$ is isomorphic to
$G/B \times G/B$. These results are due to De Concini and Procesi
(see \cite{dCP}) in the more general setting of symmetric spaces of 
adjoint type.

For an arbitrary $\lambda$, the variety $X_\lambda$ is usually
singular; its normalization $X_\lambda^\n$ only depends on 
the type of $\lambda$ (criteria for normality and smoothness of
$X_\lambda$ are obtained in \cite[Thm.~A, Thm.~B]{BGMR}).
The homomorphism $\varphi_\lambda$ extends to a 
$G \times G$-equivariant morphism $X \to X_\lambda$ that
we shall still denote by $\varphi_\lambda$. The pull-backs 
$\sL_X(\lambda) := \varphi_\lambda^* \sO_{\bP \End V(\lambda)}(1)$,
$\lambda \in \Lambda^+$, are exactly the globally generated
line bundles on $X$; moreover, $\sL_X(\lambda)$ is ample
if and only if $\lambda$ is regular. In particular,
$X$ admits a unique minimal ample line bundle, namely,
$\sL_X(\rho)$. The assignment $\lambda \mapsto \sL_X(\lambda)$
extends to an isomorphism 
$\Lambda \stackrel{\cong}{ \to} \Pic(X)$
(see e.g. \cite[Prop.~6.1.11]{BK}).

We shall index the boundary divisors so that 
$\sO_X(D_i) = \sL_X(\alpha_i)$ for $i = 1, \ldots, \ell$. 
Thus, $D_i$ comes with a homogeneous fibration over 
$G/Q_i \times G/P_i$, where $P_i$ denotes the maximal parabolic 
subgroup associated with the fundamental weight $\varpi_i$, 
and $Q_i$ denotes its opposite.

We now investigate the closures in $X$ of the images of the 
multiplicative one-parameter subgroups $\eta : \bG_m \to G$, 
and determine the degrees of the line bundles $\sL_X(\lambda)$
on these rational curves. Recall that every such $\eta$ is conjugate 
in $G$ to a dominant one-parameter subgroup of $T$. Since 
$\sL_X(\lambda)$ is invariant under conjugation, we may 
assume that $\eta : \bG_m \to T$ is dominant. We may further
assume that $\eta$ is indivisible in the lattice of one-parameter 
subgroups of $T$; equivalently, $\eta$ is an isomorphism over 
its image.

\begin{lemma}\label{lem:mult}
Let $\eta : \bG_m \to T$ be a dominant indivisible
one-parameter subgroup, and denote again by 
$\eta : \bP^1 \to X$ the corresponding morphism, with image 
$C_\eta$. Then

\item{\rm \rm (i)} $\eta(0) \in \sO_I$, where 
$I := \{ i ~\vert~ 1 \leq i \leq \ell, 
\langle \alpha_i, \eta \rangle \neq 0 \}$.
Also, $\eta(\infty) \in \sO_J$, where 
$J := \{ j ~\vert~ 1 \leq j \leq \ell, 
\langle w_0 \alpha_j, \eta \rangle \neq 0 \}$. 

\item{\rm \rm (ii)} $C_\eta$ is smooth if and only if there exists $i$ 
such that $\langle \alpha_i, \eta \rangle = 1$. 

\item{\rm \rm (iii)} $\deg \eta^* \sL_X(\lambda) = 
\langle \lambda - w_0\lambda, \eta \rangle$
for any $\lambda \in \Lambda$. 
\end{lemma}

\begin{proof}
Recall that the closure $\overline{T}$ of $T$ in $X$ contains 
an open affine $T \times T$-subset $\overline{T}_0$, isomorphic
to the affine space $\A^{\ell}$ on which $T \times T$
acts linearly with weights 
$(\alpha_1,-\alpha_1), \ldots,(\alpha_\ell, -\alpha_\ell)$.
Moreover, $\overline{T}$ is the union of the conjugates
$w \overline{T}_0 w^{-1}$, where $w \in W$, and the pull-back of
$\sL_X(\lambda)$ to $w \overline{T}_0 w^{-1}$ has a trivializing
section (as a $T \times T$-linearized line bundle) of 
weight $(w\lambda, - w \lambda)$; see 
\cite[Lem.~6.1.6, Prop.~6.2.3]{BK} for these results.
 
Since $\eta$ is dominant, it extends to a morphism
$\A^1 \to \overline{T}_0$, 
$t \mapsto (t^{\langle \alpha_1, \eta \rangle}, \ldots,
t^{\langle \alpha_\ell, \eta \rangle})$, where $\overline{T}_0$
is identified with $\A^\ell$ as above. In particular, 
$\eta(0) \in \overline{T}_0$; moreover, $C_\eta$ is smooth
at $\eta(0)$ if and only if there exists $i$ such that
$\langle \alpha_i, \eta \rangle = 1$. Also, the one-parameter
subgroup $-w_0 \eta  : \bG_m \to T$, $t \mapsto w_0(t^{-1})$ 
is dominant, and hence extends to a morphism
$\A^1 \to \overline{T}_0$,
$t \mapsto (t^{\langle -w_0 \alpha_1), \eta \rangle}, \ldots,
t^{\langle -w_0 \alpha_\ell, \eta \rangle})$. Since $-w_0$ permutes 
the simple roots, it follows that 
$\eta(\infty) = (-\eta)(0) \in w_0 \overline{T}_0 w_0$; 
moreover, $C_\eta$ is smooth at $\eta(\infty)$ if and only if 
there exists $i$ such that $\langle \alpha_i, \eta \rangle = 1$.
This proves (i) and (ii).

For (iii), note that $\eta^* \sL_X(\lambda)$ 
is a $\bG_m$-linearized line bundle on $\bP^1$ with weights
$\langle \lambda, \eta \rangle$ at $0$ and  
$\langle w_0 \lambda, \eta \rangle$ at $\infty$. Since
the degree of such a line bundle is the difference of
its weights, this yields our assertion.
\end{proof}

Next, we consider the curves obtained as closures
of the additive one-parameter subgroups associated with 
the roots; these curves are clearly rational, and smooth 
by Lemma \ref{lem:smooth}.
For any $\alpha \in R$, we denote by $U_\alpha$ the corresponding 
root subgroup of $G$ (with Lie algebra the root subspace 
$\fg_\alpha \subset \fg$) 
and by $C_\alpha$ the closure of $U_\alpha$ in $X$. 
Since $\alpha$ is conjugate in $W$ to a dominant root,
we may assume $\alpha$ dominant; then $\alpha$ is either
the highest root or the highest short root of an irreducible
component of the root system $R$. The highest roots are of 
special interest in view of the following:

\begin{lemma}\label{lem:stable} 
The irreducible $B$-stable curves through 
$x$ (the neutral element of $G$)
are exactly $C_{\theta_1}, \ldots, C_{\theta_m}$, where 
$\theta_1, \ldots, \theta_m$ denote the highest roots of $R$.
\end{lemma}

\begin{proof}
The isotropy representation of the homogeneous space
$(G \times G)/\diag(G)$ is the adjoint representation of $G$ 
in $\fg$; this is the direct sum of irreducible representations 
with highest weights $\theta_1,\ldots,\theta_m$. Since these 
weights are linearly independent, the assertion follows from 
Lemma \ref{lem:curves}.
\end{proof}

To analyze further these curves associated with the
highest roots, we may and shall assume that $G$ is
simple, by using the decompositions (\ref{eqn:G}) and
(\ref{eqn:X}); we denote the highest root by $\theta$.
We shall need the following observation:

\begin{lemma}\label{lem:roots}
\item{\rm (i)} We have $\langle \alpha, \theta^\vee \rangle = 0$ 
or $1$ for any $\alpha \in R^+$ such that $\alpha \neq \theta$.

\item{\rm (ii)} If $G$ is of type $A_\ell$, then 
$\langle \alpha_i, \theta^\vee \rangle = 1$ if and only if 
$i = 1$ or $i = \ell$. If $G$ is not of type $A$, then 
there is a unique simple root $\alpha_{i_0}$ such that
$\langle \alpha_{i_0}, \theta^\vee \rangle = 1$.
\end{lemma}

\begin{proof}
(i) follows from \cite[VI.1.8, Prop.~25]{Bourbaki}. 

(ii) is checked by inspection (see also \cite[Table I]{CP}).
\end{proof}

We now obtain an analogue of Lemma \ref{lem:mult} for the curve
$C_\theta$:

\begin{lemma}\label{lem:add}
\item{\rm (i)} For any $\lambda \in \Lambda$, we have
$\sL_X(\lambda) \cdot C_\theta = 
\langle \lambda, \theta^\vee \rangle$.

\item{\rm (ii)} Let $y$ denote the unique point of 
$C_\theta \setminus U_\theta$ (the point at infinity of $C_\theta$).
Then $y \in \sO_{1,\ell}$ if $G$ is of type $A_\ell$,
and $y \in \sO_{i_0}$ otherwise. 
\end{lemma}

\begin{proof}
(i) The coroot $\theta^\vee$ is a dominant one-parameter
subgroup of $T$, and satisfies $w_0  \theta^\vee = - \theta^\vee$.
Thus, 
$\deg (\theta^\vee)^* \sL_X(\lambda) = 
2 \langle \lambda,  \theta^\vee \rangle$ 
by Lemma \ref{lem:mult}. We now relate $C_\theta$ with the curve 
$C_{\theta^\vee}$, the image of $\theta^\vee : \bP^1 \to X$. 
For this, we consider the subgroup $G_\theta \subset G$ 
generated by $U_\theta$ and $U_{-\theta}$. Then $G_\theta$ 
is the image of a homomorphism $h : \SL_2 \to G$; in particular,
$G_\theta$ is semisimple with maximal torus the image of 
$\theta^\vee$ (viewed as a one-parameter subgroup of $T$). 
Recall from Lemma \ref{lem:roots} that 
$\langle \alpha_i, \theta^\vee \rangle = 1$ for some $i$.
It follows that $\theta^\vee$ is indivisible, and
(using Lemma \ref{lem:mult} again) $C_{\theta^\vee}$
is smooth. Also, $\theta^\vee(-1) = h(-\id)$ is a central element 
of $G_\theta$, and is non-trivial since $\theta^\vee$ is indivisible. 
So $G_\theta$ is isomorphic to $\SL_2$ via $h$, and 
the closure $\overline{G_\theta} \subset X$ may be viewed 
as an $\SL_2 \times \SL_2$-equivariant compactification of 
$\SL_2$. Using the local structure theorem as in 
\cite[Sec.~6]{Timashev}, one obtains that 
$\overline{G_\theta}$ is smooth.

Consider the standard compactification $\overline{\SL_2}$, 
the closure of $\SL_2 \subset \End(\C^2)$ in 
$\bP( \End(\C^2) \oplus \C)$. 
This is a smooth quadric with equation $ad - bc = z^2$, 
and the boundary $\overline{SL_2} \setminus \SL_2$ 
is the quadric $(ad - bc = z = 0)$, an irreducible divisor 
on which $\SL_2 \times \SL_2$ acts transitively with isotropy
group a Borel subgroup. In view of Lemma \ref{lem:comp}, 
it follows that the equivariant rational map 
$\overline{\SL_2} \dasharrow \overline{G_\theta}$
is a bijective morphism, and hence an isomorphism. 
The closure in $\overline{\SL_2}$ of the diagonal torus 
is the conic $(ad= z^2, b = c = 0)$, while the closure 
of the upper triangular unipotent subgroup
is the line $(a = d = z, b = 0)$. Thus, $2 C_\theta$
is rationally equivalent to $C_{\theta^\vee}$ in
$\overline{G_\theta}$; this implies our formula. 

(ii) In view of (i) and Lemma \ref{lem:roots},
we have $D_{i_0} \cdot C_\theta = 1$ and $D_i \cdot C_\theta = 0$
for $i \neq i_0$, when $G$ is not of type $A$. It follows
that $y$ sits in $D_{i_0}$ and in no other boundary divisor.
Also, $D_1 \cdot C_\theta = D_\ell \cdot C_\theta = 1$ and 
$D_i \cdot C_\theta = 0$ for $1 < i < \ell$, when $G$
is of type $A$. In that case, $y$ sits in $D_1 \cap D_\ell$
and in no other boundary divisor. This yields our assertion 
by using the fact that 
$\sO_I = (\bigcap_{i \in I} D_i)  
\setminus (\bigcup_{j \notin I} D_j)$
for any $I \subset \{ 1 \ldots, \ell\}$.
\end{proof}

\begin{remark}\label{rem:short}
We still assume that $G$ is simple, and in addition that it
has two root lengths; we denote by $\theta_s$ the highest 
short root. Then we have 
$\sL_X(\lambda) \cdot C_{\theta_s} = 
\langle \lambda, \theta_s^\vee \rangle$ 
for any $\lambda \in \Lambda$. Moreover, the point
at infinity $y_s$ of $C_{\theta_s}$ sits in a 
$G \times G$-orbit $\sO_i$ of codimension $1$.

Indeed, the root system $R$ is of type 
$B_\ell$, $C_\ell$, $F_4$ or $G_2$. 
In type $B_\ell$, where $\ell \geq 3$, one has 
$\langle \alpha_1, \theta_s^\vee \rangle = 2$
and $\langle \alpha_i, \theta_s^\vee \rangle = 0$ for all 
$i \geq 2$. In particular, $\theta_s^\vee = 2 \eta$ for an
indivisible one-parameter subgroup $\eta: \bG_m \to T$. 
In that case, the subgroup $G_{\theta_s} \subset G$ generated 
by $U_{\theta_s}$ and $U_{-\theta_s}$ is isomorphic to $\PGL_2$, 
and one obtains that $C_{\theta_s}$ and $C_\eta$ are lines in 
$\overline{G_{\theta_s}} \cong \bP^3$. Then one concludes by using
Lemma \ref{lem:mult}. Note that $D_1 \cdot C_{\theta_s} = 2$
and $D_i \cdot C_{\theta_s} = 0$ for all $i \geq 2$; thus,
$y_s \in \sO_1$.

In type $C_\ell$ ($\ell \geq 3$), one has   
$\langle \alpha_2, \theta_s^\vee \rangle = 1$
and $\langle \alpha_i, \theta_s^\vee \rangle = 0$ for all 
$i \neq 2$. In that case, $\theta_s^\vee$ is indivisible 
and the assertion follows by arguing as in
the proof of Lemma \ref{lem:add}; here $y_s \in \sO_2$.

The cases of $F_4$ and $G_2$ are similar; in the former
one, one obtains $\langle \alpha_4, \theta_s^\vee \rangle = 1$
and $\langle \alpha_i, \theta_s^\vee \rangle = 0$ for all other 
$i$, so that $y_s \in \sO_4$. For $G_2$, one has
$\langle \alpha_1, \theta_s^\vee \rangle = 1$
and $\langle \alpha_2, \theta_s^\vee \rangle = 0$, so that
$y_s \in \sO_1$.
\end{remark}

\begin{remark}\label{rem:lines}
We have $L \cdot C \geq \ell$ for any ample line bundle $L$ 
on $X$ and any curve $C$ through $x$. In particular, 
$X$ is not covered by lines when $\ell \geq 2$.

Indeed, note first that $C$ is rationally equivalent to an effective 
$1$-cycle stable by $B$ whose support contains $x$. To see this,
consider the Chow variety ${\rm Chow}_x$ of effective $1$-cycles 
through $x$, which has a natural $B$-action.  The closure of 
the $B$-orbit of $[C] \in {\rm Chow}_x$ is then a rational variety, 
on which $B$ has a fixed point $[\gamma]$. Thus, the curve $C$ 
is rationally equivalent to the effective $1$-cycle $\gamma$.

Next, by Lemma \ref{lem:stable} and the description
of ample line bundles on $X$, it suffices to check that
$\sL_X(\rho) \cdot C_\theta \geq \ell$, i.e., 
$\langle \rho, \theta^\vee \rangle \geq \ell$
in view of Lemma \ref{lem:add}.
But this follows from the fact that
$\theta^\vee = \sum_{i= 1}^\ell m_i \alpha_i^\vee$, 
where the $m_i$ are positive integers.

Yet $X$ is covered by strict transforms of lines under some
morphism $\varphi_\lambda : X \to X_\lambda$, unless $G$ is of
type $E_8$ (recall that these morphisms yield all contractions
from $X$). Indeed, by the above argument, $X_\lambda$ is
covered by lines if and only if 
$\langle \lambda, \theta^\vee \rangle = 1$.
In turn, this is equivalent to $\lambda = \varpi_i$, 
where $m_i = 1$. The list of these fundamental weights 
according to the type of $G$ is as follows:

$A_\ell$ : all

$B_\ell$ : $\varpi_1$, $\varpi_\ell$

$C_\ell$ : all

$D_\ell$ : $\varpi_1$, $\varpi_{\ell - 1}$, $\varpi_\ell$

$E_6$ :  $\varpi_1$, $\varpi_6$

$E_7$ : $\varpi_7$

$E_8$ : none

$F_4$ : $\varpi_4$

$G_2$ : $\varpi_1$.
 
By \cite[Thm.~A]{BGMR}, $X_\lambda$ is normal in all these cases,
except for $(B_\ell,\varpi_1)$ and $(C_\ell,\varpi_\ell)$. Also,
$X_\lambda$ is smooth in the cases $(A_\ell, \varpi_1)$, 
$(A_\ell, \varpi_\ell)$, $(B_\ell,\varpi_\ell)$ and $(G_2,\varpi_1)$
only, as follows from \cite[Thm.~B]{BGMR}.

The weights in the above list include the minuscule weights,
i.e., those fundamental weights $\varpi_i$ such that 
$\alpha_i^\vee$ has coefficient $1$ in the highest coroot,
$\theta_s^\vee$. 
When $G$ is simply laced (i.e., all roots have the same length), 
one has $\theta_s^\vee = \theta$ and hence one gets exactly 
the minuscule weights. In the non-simply laced case, 
one obtains additional weights for each type.
\end{remark}

\section{Proofs of Theorem \ref{thm:main} 
and Proposition \ref{prop:prod}}
\label{sec:proofs}

\begin{proof}[Proof of Theorem \ref{thm:main}]
By Proposition \ref{prop:irr} and Lemma \ref{lem:stable},
$\sK$ is unique and $\sK_x$ is smooth and connected.
The assertion that $\tau$ is an isomorphism will be checked 
separately in types other than $A$ (case (ii)) and in type $A$ 
(case (iii)).

By Proposition \ref{prop:irr} again, $G$ has a unique closed
orbit in $\sK_x$, namely, that of $C_\theta$. By Lemma 
\ref{lem:stable}, the isotropy group of $C_\theta$ in $G$ is
the parabolic subgroup $P_\theta$, the stabilizer of the 
highest weight line $[\fg_\theta] \in \bP \fg$. Thus, 
$G \cdot C_\theta \cong G/P_\theta$ is the adjoint variety 
in the sense of \cite{CP}; it may be viewed as the 
projectivization $\bP \sO_{\min}$, where $\sO_{\min}$ denotes 
the minimal (non-zero) nilpotent orbit, i.e., the $G$-orbit 
of any non-zero point in $\fg_\theta$. By Lemma \ref{lem:roots}, 
$P_\theta$ is the maximal parabolic subgroup $P_{i_0}$ 
associated with the simple root $\alpha_{i_0}$ when $G$ 
is not of type $A$;  in type $A_\ell$, we have 
$P_\theta = P_1 \cap P_\ell$. 
 
The dimension of $\sK_x$ is given 
by $-K_X \cdot C_\theta -2 $ by \cite[Prop.~2.3]{Hwang}. 
In view of \cite[Prop.~6.1.11]{BK}, we have 
$\sO_X(-K_X) = \sL_X(\kappa)$, where 
\[ \kappa := 2 \rho + \sum_{i=1}^\ell \alpha_i. \]
By Lemma \ref{lem:add}, it follows that
\[ \dim \sK_x = \langle \kappa, \theta^\vee \rangle -2. \]

\begin{lemma} \label{lem:dim}
(i) If $\fg$ is of type $A_\ell$ ($\ell \geq 2$), then 
$\langle \kappa, \theta^\vee \rangle = \dim \bP\sO_{\min} + 3$.

(ii) If $\fg$ is not of type $A$, then  
$\langle \kappa, \theta^\vee \rangle = \dim \bP\sO_{\min} +2$.
\end{lemma}

\begin{proof}
The dimension of $G/P_\theta$ is the number of positive roots 
non-orthogonal to $\theta$. In view of Lemma \ref{lem:roots} 
(i), this implies that
\[
\dim \bP\sO_{\min} = 
(\sum_{\alpha \in R^+} \langle \alpha, \theta^\vee \rangle) -1 
= \langle 2 \rho, \theta^\vee \rangle -1. 
\]
Also, by Lemma \ref{lem:roots} (ii), we have 
$\langle \sum_i \alpha_i, \theta^\vee \rangle = 2$ 
for $G$ of type $A_\ell$ ($l\geq 2$), and 
$\langle \sum_i \alpha_i, \theta^\vee \rangle = 1$ otherwise. 
\end{proof}

When $G$ is not of type $A$, Lemma \ref{lem:dim} yields that
$\dim \sC_x = \dim \sK_x = \dim \bP \sO_{\min}$. 
Since $\bP \sO_{\min}$ is the closed orbit in $\bP \fg$, 
it follows that $\sC_x = \bP \sO_{\min}$. 
As the rational map $\tau : \sK_x \dasharrow \sC_x$ is 
$G$-equivariant, it must be an isomorphism. This completes the
proof of Theorem \ref{thm:main} in that case.

When $G$ is of type $A_\ell$, we may assume that 
$G = \PGL_{\ell + 1}$ and $T$ (resp. $B$) is the image in $G$ 
of the group of diagonal (resp. upper triangular) invertible 
matrices. Let $\eta : \bG_m \to T$ 
denote the image of the one-parameter subgroup 
$t \mapsto \diag(t,1,\ldots,1)$ of $\GL_{\ell + 1}$. Then $\eta$ 
is dominant and $\eta - w_0 \eta = \theta^\vee$. 
So Lemma \ref{lem:add} yields that
$\deg \eta^* \sL_X(\lambda) =
\langle \lambda, \theta^\vee \rangle = 
\sL_X(\lambda) \cdot C_\theta$. As a consequence, the image
$C_\eta$ of $\eta$ in $\RatCurves^\n(X)$ has minimal degrees.
Since $C_\eta$ meets $X^0$, it sits in a covering family
of rational curves of minimal degrees. But every such 
family is minimal, since its image in the Chow variety is closed
(see the proof of Theorem IV.2.4 in \cite{Kollar} for details).
Thus, $C_\eta \in \sK_x$ by the uniqueness of $\sK$.
The stabilizer of $C_\eta$ under the action of $G$ 
is the normalizer of the image of $\eta$, that is, 
the image in $\PGL_{\ell + 1}$ of 
$\GL_\ell \subset \GL_{\ell + 1}$. 
Thus, the $G$-orbit of $C_\eta$ has dimension 
$2 \ell = \langle \kappa, \theta^\vee \rangle - 2 = \dim \sK_x$.
Since $\sK_x$ is irreducible, it is the $G$-orbit closure 
of $C_\eta$, an equivariant compactification of 
$\PGL_{\ell + 1}/\GL_\ell$. This homogeneous variety has another 
equivariant compactification, 
$\bP^\ell \times (\bP^\ell)^*$, with homogeneous and
irreducible boundary: the incidence variety. 
By Lemma \ref{lem:comp} and the smoothness of $\sK_x$, it follows 
that $\sK_x = \bP^\ell \times (\bP^\ell)^*$.
Moreover, the tangent line of $C_\eta$ at $x$ is
spanned by the differential of $\eta$ at $1$, which is 
the image in $\fg = \End(\C^{\ell+1})/\C \id$ of the projection of
$\C^{\ell + 1}$ onto the first coordinate line. The orbit of that
line is open in $\sC_x$, which is thus the image of
$\bP^\ell \times (\bP^\ell)^*$ (the variety of projections onto lines) 
in $\bP \fg$ under the rational map 
$\bP \End(\C^{\ell+1}) \dasharrow \bP (\End(\C^{\ell+1})/\C \id) = \bP \fg$
(the projection with center $[\id]$). Since $\ell \geq 2$, 
the secant variety of $\bP^\ell \times (\bP^\ell)^*$ does not
contain $[\id]$; hence $\bP^\ell \times (\bP^\ell)^*$ is sent
isomorphically to its image in $\bP \fg$. Using Lemma
\ref{lem:comp} again, we conclude that 
$\tau : \sK_x \dasharrow \sC_x$ is an isomorphism.
\end{proof}

\begin{proof}[Proof of Proposition \ref{prop:prod}]
Let $\sK$ be a family of minimal rational curves on $X$. 
Then $\sK_x$ contains some irreducible $B$-stable curve 
$C_{\theta_i}$ by Lemma \ref{lem:stable}. As a consequence, 
every curve $C$ in $\sK_x$ satisfies 
$L \cdot C = L \cdot C_{\theta_i}$ for any line bundle 
$L$ on $X$. Also, recall the decomposition (\ref{eqn:X}),
where $C_{\theta_i} \subset X_i$. Thus, $L \cdot C = 0$ 
whenever $L$ is the pull-back of a ample line bundle on 
some factor $X_j$ with $j \neq i$. Thus, $C$ is contracted by 
the corresponding projection $X \to X_j$, i.e., 
$C$ is contained in $X_i$.
\end{proof}

\section{Geometric constructions of the minimal rational curves}
\label{sec:geom}

Consider first a simple group $G$ of adjoint type $A_\ell$, 
where $\ell \geq 2$; then $G = \PGL(V)$, where 
$\dim(V) = \ell + 1$. The wonderful compactification 
admits a $G \times G$-equivariant morphism $f : X \to \bP \End(V)$,
which restricts to an isomorphism over the open orbit $\PGL(V)$. 
We recall a classical description of $f$ as a sequence of blow-ups 
with smooth centers, see e.g. \cite[Thm.~1]{Vainsencher}.

The $G \times G$-orbit closures in $\bP \End(V)$ are exactly 
the loci $Y_i$ of endomorphisms of rank $\leq i$, where 
$i = 1, \ldots, \ell + 1$; we have
$Y_1 \subset \cdots \subset Y_\ell \subset  Y_{\ell + 1} 
= \bP \End(V)$ and $Y_i$ is the singular locus of $Y_{i + 1}$.
Put $X_0 = \bP \End(V)$ (so that $X_0 = X_{\varpi_1}$) and let  
$\pi_i: X_i \to X_{i-1}, i=1, \ldots,\ell -1$, 
be the blow-up of $X_{i-1}$ along the strict transform of $Y_i$.
Then $X = X_{\ell -1}$ and $f$ is the composition 
$\pi_{\ell -1} \circ \cdots \circ \pi_1$.

In particular, $Y_1$ is the closed $G \times G$-orbit in 
$\bP \End(V)$, the image of the Segre embedding
$\bP V \times \bP V^* \to \bP(V \otimes V^*) = \bP \End(V)$.
Moreover, the closed orbit $\bP \sO_{\min}$ is the 
incidence variety in $\bP V \times \bP V^*$, and
$Y_1 \setminus \bP \sO_{\min}$ is a unique orbit of $G \times G$.

\begin{proposition}\label{prop:A}
With the above notation, the family of minimal rational
curves $\sK_x$ consists of the strict transforms in $X$ of
the lines in $\bP \End(V)$ through $x$ that intersect $Y_1$.  
\end{proposition}

\begin{proof}
Consider the lines in $\bP \End(V)$ that intersect both the open
orbit $\PGL(V)$ and the closed orbit $Y_1$. Since $\ell \geq 2$, 
every such line meets $Y_1$ at a unique point. By 
\cite[Prop.~9.7]{FH}, the strict transforms of these lines form 
a family of minimal rational curves $\sK_1$ on 
$X_1 = {\rm Bl}_{Y_1}(\bP \End(V))$, whose VMRT at the point $x$ 
is the image of $Y_1 \subset \bP \End(V)$ under 
the linear projection $\bP \End(V) \dasharrow \bP (\End(V)/[\id])$. 
Also, this projection maps $Y_1$ isomorphically to its image.

Since the join variety $J(Y_1, Y_i)$ is contained in $Y_{i+1}$,
the lines in $\bP \End(V)$ intersecting $\PGL(V)$ and $Y_1$ 
are disjoint from $Y_i \setminus Y_1$ for all $i=2, \ldots, \ell-1$. 
Hence the strict transforms of these lines are disjoint from 
the strict transforms of $Y_i, i=2, \ldots, \ell-1$. It follows 
that the strict transforms of these lines in $X$ form a family 
$\sK$ of minimal rational curves with the same VMRT as that on $X_1$. 
This concludes the proof by the uniqueness of $\sK$ (Proposition \ref{prop:irr}).
\end{proof}

Next, consider a simple group $H$ of adjoint type $C_\ell$, i.e.,
$H = \PSp(V)$, where $\dim(V) = 2 \ell$. Then $H$ is the fixed
point subgroup of an involutive automorphism $\sigma$ of
$\PGL(V) =: G$. Choose a maximal torus $T_H \subset H$; then 
the centralizer of $T_H$ in $G$ is a maximal torus, $T$. Moreover,
$T$ is contained in a $\sigma$-stable Borel subgroup $B \subset G$,
and $B_H := B \cap H$ is a Borel subgroup of $H$; the roots of
$(H,T_H)$ are exactly the restrictions of the roots of $(G,T)$
(see \cite[VIII.13.3]{Bourbaki}). Consider the wonderful 
compactifications $H \subset Y$ and $G \subset X$. Then $\sigma$ 
extends to a unique involution of $X$, still denoted by $\sigma$. 
Moreover, $Y$ may be identified with a component of the fixed locus
$X^\sigma$, in view of \cite[Thm.~4.7]{Kannan}; this identifies
the base point $x \in X$ with that of $Y$.

\begin{proposition}\label{prop:C}
Under the above identification, the family of minimal rational
curves $\sK_{Y,x}$ is identified with $\sK_{X,x}^\sigma$. The latter 
is the image of the closed immersion 
$\iota : \bP V \to \bP V \times \bP V^*$, 
$[v] \mapsto ([v],[v^*])$, where $v \mapsto v^*$ denotes the
isomorphism $V \to V^*$ associated with the symplectic form
defining $\PSp(V)$.
\end{proposition}

\begin{proof}
By construction, $\sigma$ acts on the root system $R$
and fixes the highest root $\theta$. One checks 
that $U_\theta \subset H$, and hence $C_\theta \subset Y$
with the notation of Lemma \ref{lem:stable}. It follows
that $\sK_{Y,x} \subset \sK_{X,x}$. Also, $\sigma$ acts
on $\sK_{X,x}$ and fixes $\sK_{Y,x}$ pointwise.
Under the isomorphism $\sK_{X,x} \cong \bP V \times \bP V^*$,
the action of $\sigma$ is identified with
$([v], [f]) \mapsto ([f^*], [v^*])$. Thus, $\sK_{Y,x}$
is contained in the image of $\iota$. Since $H$ acts
transitively on $\bP V$, we must have 
$\sK_{Y,x} = \iota( \bP V)$.
\end{proof}

\begin{remark}\label{rem:B}
We may consider a similar construction in type $B_\ell$,
$\ell \geq 2$, that is, $H = \SO(V) \subset \PGL(V) = G$, where 
$\dim(V) = 2 \ell + 1$. Then the wonderful compactification
$Y$ is still contained in $X$ as a component of the fixed
locus of an involution $\sigma$, by \cite[Thm.~4.7]{Kannan}
again. But one may check that $\sC_{Y,x}$ is not contained
in $\sC_{X,x}$ in that case. 
\end{remark}

We now obtain an alternative description of the minimal rational 
curves, for any simple group $G$ of type $\neq A$. Recall that 
$P_\theta$ is then the maximal parabolic subgroup $P_{i_0}$.
Moreover, the boundary divisor $D_{i_0}$ admits a 
$G\times G$-equivariant morphism to $G/Q_{i_0} \times G/P_{i_0}$, 
where $Q_{i_0}$ denotes the maximal parabolic subgroup
opposite to $P_{i_0}$ (note that $Q_{i_0} = w_0 P_{i_0} w_0$,
since $w_0(\alpha_{i_0}) = - \alpha_{i_0}$). We denote by 
\[ p : \sO_{i_0} \to G/P_{i_0} = G/P_\theta \]
the resulting projection; then $p$ is $G$-equivariant 
for the right action of $G$ on $D_{i_0}$.

\begin{proposition}\label{prop:lim}
\item{\rm (i)} Every $C \in \sK_x$ intersects 
$\sO_{i_0}$ transversally at a unique point, denoted by $f(C)$.

\item{\rm (ii)} The assignment 
$C \mapsto p(f(C))$ yields a $G$-equivariant
isomorphism $\sK_x \stackrel{\cong}{\to} G/P_\theta$.

\item{\rm (iii)} The morphism 
$\varphi_\theta : X \to X_\theta \subset \bP \End(\fg)$
contracts $D_{i_0}$ to the closed $G \times G$-orbit, 
isomorphic to $G/P_\theta \times G/P_\theta$.

\item{\rm (iv)} For any $C \in \sK_x$, the image
$\varphi_\theta(C)$ is a smooth conic in $\bP \End(\fg)$,
which intersects $G/P_\theta \times G/P_\theta$
at a unique point, $f_\theta(C) \in \diag(G/P_\theta)$.
Moreover, the assignment $C \mapsto f_\theta(C)$,
$\sK_x \to G/P_\theta$ gives back the isomorphism $p \circ f$.
\end{proposition} 
 
\begin{proof}
(i) We have $D_i \cdot C = D_i \cdot C_\theta$ for 
$i = 1,\ldots,\ell$. This implies our assertion
by arguing as in the proof  of Lemma \ref{lem:add}.

(ii) We claim that the assignment 
$C \mapsto f(C)$ is a morphism. To see this,
consider the subvariety $Y$ of the universal family 
$\sU_x \subset X \times \sK_x$ consisting of 
those pairs $(y,C)$ such that $y \in \sO_{i_0}$. 
Then the projection $\rho : \sU_x \to \sK_x$ 
restricts to a bijective morphism $Y \to \sK_x$
in view of (i). Since $\sK_x$ is smooth, 
this yields our claim.

By that claim, the assignment $C \mapsto p(f(C))$
yields a morphism, which is clearly equivariant.
Since $\sK_x \cong G/P_\theta$ by Theorem 
\ref{thm:main}, we thus obtain an equivariant
endomorphism of $G/P_\theta$. But the identity
is the unique such endomorphism, since 
$P_\theta$ is its own normalizer in $G$.

(iii) Recall that the orbit $\sO_{i_0}$ contains a point 
$x_{i_0}$ fixed by $R_u(P_\theta) \times R_u(P_\theta)$,
where $R_u$ stands for the unipotent radical.
Thus, $\varphi_\theta(x_{i_0}) \in \bP \End(\fg)
\cong \bP (\fg^* \otimes \fg)$ is also fixed
by $R_u(P_\theta) \times R_u(P_\theta)$. But we have
\[
\bP (\fg^* \otimes \fg)^{R_u(P_\theta) \times R_u(P_\theta)} 
= \bP ((\fg^*)^{R_u(P_\theta)}) \times \bP(\fg^{R_u(P_\theta)}).
\]
Moreover, $\fg^{R_u(P_\theta)}$ is a representation
of the Levi subgroup $L_\theta$ with a unique highest weight
line: $\fg_\theta$, stable by $L_\theta$. It follows that
$\fg^{R_u(P_\theta)}$ is just the highest weight line 
$\fg_\theta$; likewise, $(\fg^*)^{R_u(P_\theta)}$ is the
highest weight line. Thus, $\varphi_\theta(x_{i_0})$ sits 
in the closed $G\times G$-orbit; this yields our assertion.

(iv) We have
$\sL_X(\theta) \cdot C = \sL_X(\theta) \cdot C_\theta
= \langle \theta, \theta^\vee \rangle = 2$. Thus,
$\varphi_\theta(C)$ is a conic. Moreover, $\varphi_\theta(C_\theta)$
is smooth by Lemma \ref{lem:smooth}, and intersects the
closed $G \times G$-orbit at a unique point, fixed by 
$\diag(B)$; thus, this point is identified with the base
point of $G/P_\theta \times G/P_\theta$. The remaining 
assertions follow by using (i) and Theorem \ref{thm:main}.
\end{proof}

\begin{remark}\label{rem:lim}
By \cite[Thm.~A, Thm.~B]{BGMR} again, $X_\theta$ is smooth in
types $A_\ell$ and $G_2$; normal and singular in types $C_\ell$
($\ell \geq 3$), $D_\ell$, $E_6$, $E_7$ and $E_8$; non-normal in 
the remaining types, $B_\ell$ and $F_4$. 

Also, the above constructions can be adapted to the type $A_\ell$: 
then every $C \in \sK_x$ intersects transversally the 
boundary divisor $D_1$ at a unique point, and likewise for 
$D_\ell$. Using the projections $D_1 \to G/P_1 \cong \bP^\ell$
and $D_\ell \to G/P_\ell \cong (\bP^\ell)^*$, one obtains 
a $\PGL_{\ell + 1}$-equivariant morphism 
$f: \sK_x \to \bP^\ell \times (\bP^\ell)^*$,
which is thus identified with an equivariant
endomorphism of $\bP^\ell \times (\bP^\ell)^*$.
One concludes as above that $f$ is an isomorphism,
by using the fact that the isotropy group of the open 
orbit (the image of $\GL_\ell$ in $\PGL_{\ell + 1}$) 
is its own normalizer. But this construction just yields a variant 
of that considered in Proposition \ref{prop:A}.
\end{remark}

\section{An application}
\label{sec:appl}

We return to the general setup of the introduction, 
where $X$ is a projective uniruled manifold and 
$\sK \subset \RatCurves^\n(X)$ a family of minimal 
rational curves. Let $\sC$ be the subvariety of 
the projectivized tangent bundle $\bP T(X)$ formed 
by VMRTs of $\sK$, namely the fiber of $\sC \to X$ 
at a general point $x\in X$ is the VMRT 
$\sC_x \subset \BP T_x(X)$. 
A germ of holomorphic vector field $v$ at
$x$ is said to {\em preserve $\sC$} if the local
one-parameter family of biholomorphisms integrating 
$v$ lifts to local biholomorphisms of $\BP T(X)$ 
preserving $\sC$. The set of all such germs forms 
a Lie algebra, called the
{\em Lie algebra of infinitesimal automorphisms of the VMRT
structure} $\sC$ at $x$, to be denoted by $\aut(\sC, x)$.
The manifold $X$ is said to have the 
{\em Liouville property with respect to $\sK$} 
if every infinitesimal automorphism of $\sC$ at a general point 
$x \in X$ extends to a global holomorphic vector field on 
some Zariski open subset of $X$. In particular, if 
$\aut(\sC, x) = \aut(X)$ (the Lie algebra of global holomorphic
vector fields), then $X$ has the Liouville property.

As an application of Theorem \ref{thm:main}, we shall prove:

\begin{proposition}\label{prop:liouville}
Let $X$ be the wonderful compactification of a simple
algebraic group $G$ of adjoint type. Assume that $G$ is not of 
type $A_1$ or $C$. Then $X$ has the Liouville property.
\end{proposition}

\begin{proof}
Assume first that $G$ is of type $A_\ell$ $(\ell \geq 2)$. Let 
$Y_1 = \BP V \times \BP V^* \subset \BP {\rm End}(V)$ be the Segre embedding. 
Then the blow-up ${\rm Bl}_{Y_1}(\BP {\rm End}(V))$ has the Liouville property by  
\cite[Thm.~9.6]{FH}. It follows then from \cite[Prop.~9.4]{FH}
that $X$ has the Liouville property as well.

Next, assume $G$ is not of type $A$ nor of type $C$. 
By Theorem \ref{thm:main}, the VMRT $\sC_x \subset \BP T_x(X)$  of $X$ 
at a general point $x$ is $\BP \0_{\min} = G/P_{\theta} \subset \BP \fg$. 
Using \cite[Thm.~1]{Demazure}, it follows that the natural map 
$\fg \to \aut(\sC_x)$ is an isomorphism.
Consider the affine cone $\hat{\sC_x} \subset T_x(X)$ and its Lie algebra
of linear infinitesimal automorphisms, $\aut(\hat{\sC_x})$; then
we obtain $\aut(\hat{\sC_x}) \cong \fg \oplus \C$.

We now use a theorem of Cartan and Kobayashi-Nagano 
(cf. \cite[Thm.~1.1]{FH}); to state it, we need the notion 
of prolongation. Let $V$ be a finite-dimensional vector space and 
$\fl \subset {\rm End}(V)$ a Lie subalgebra. 
The {\em $k$-th prolongation} (denoted by $\fl^{(k)}$) of $\fg$ 
is defined by $\fl^{(k)} = \Hom(\Sym^{k+1}V, V) \cap \Hom(\Sym^kV, \fl)$,
where the intersection is taken in 
$(V^*)^{\otimes(k + 1)} \otimes V = (V^*)^{\otimes k} \otimes \End(V)$.
The theorem asserts that if $V$ is an irreducible representation 
of $\fl$, and if $\fl^{(2)} \neq 0$, then 
$\fl = \fgl(V), \fsl(V), \fsp(V)$ or $\fcsp(V)$. 
Moreover, if $\fl^{(2)} = 0$ but $\fl^{(1)} \neq 0$, then 
$\fl \subset \End(V)$ is the isotropy representation of an irreducible 
Hermitian symmetric space of compact type, different from the projective
space.

Applying this result to $\aut(\hat{\sC_x}) \subset \End(\fg)$, we obtain 
that $\aut(\hat{\sC_x})^{(1)} = 0$. In view of 
\cite[Def.~5.3, Prop.~5.10]{FH},  it follows that
\begin{equation}\label{eqn:ineq} 
\dim \aut(\sC, x) \leq \dim(G) + \dim \aut(\hat{\sC}_x) = 2 n +1, 
\end{equation}
where $n := \dim(G) = \dim(X)$. 
On the other hand, we have $\aut(X) \subset \aut(\sC, x)$, since the action 
of $\Aut(X)$ preserves the VMRT structure $\sC$ on $X$. Now recall 
from \cite[Ex.~2.4.5]{Brion} that the natural map 
$\fg \oplus \fg \to \aut(X)$ is an isomorphism (in all types except $A_1$).
This gives the inequalities
$2n = \dim \aut(X) \leq \dim(\aut(\sC, x) \leq 2n +1.$

If  $\dim \aut(\sC, x) = 2n$, then $X$ has the Liouville property.  
So we may assume that $\dim \aut(\sC, x) = 2n +1$. 
Then the inequality in (\ref{eqn:ineq}) is an equality, which implies by 
\cite[Cor.~5.13]{FH} that the VMRT structure $\sC$ is locally flat, 
that is, for a general point $x \in X$, there exists an analytic
open subset $U \subset X$, an analytic open subset 
$W \subset \C^n$ and a biholomorphic map $\phi: U \to W$ 
such that the induced map $d \phi: \BP T(U) \to \BP T(W)$ sends 
$\sC|_U$ to $W \times \sC_x$.

By \cite[Prop.~5.14]{FH}, it follows that 
$\aut(\sC, x) = V \oplus \aut(\hat{\sC_x})$, 
where $V$ denotes the abelian Lie algebra of dimension 
$n$. Since $\aut(\hat{\sC_x}, x) = \fg \oplus \C$,
this gives an inclusion of Lie algebras: 
$\aut(X) = \fg \oplus \fg \subset V \oplus \fg \oplus \C$, 
which is not possible since $\fg$ is simple.
\end{proof}

Combining Proposition \ref{prop:liouville} with 
\cite[Prop.~9.5]{FH}, we obtain the following:

\begin{corollary}\label{cor:TRP}
Assume that $G$ is not of type $A_1$ or $C$.
Then for any projective variety $Y$ and any family of morphisms
$\{ f_t : Y \to X, |t| <1 \}$, where $f_0$ is surjective, 
there exists a family of automorphisms 
$\sigma_t: X \to X$ such that $f_t = \sigma_t \circ f_0$.
\end{corollary}

\begin{remark}
When $G$ is of type $A_1$, the variety $X$ is isomorphic to $\BP^3$, 
which does not have the Liouville property. The statement in 
Corollary \ref{cor:TRP} also fails in this case.  
\end{remark}

\bigskip

Michel Brion 

Institut Fourier, Universit\'e de Grenoble, France 

e-mail: Michel.Brion@ujf-grenoble.fr 

\bigskip
Baohua Fu

Institute of Mathematics, AMSS, 55 ZhongGuanCun East Road,

Beijing, 100190, China

e-mail: bhfu@math.ac.cn


\begin{thebibliography}{100}

\bibitem[Bo]{Bourbaki}
Bourbaki, Nicolas,
\emph{Groupes et alg\`ebres de Lie}, 
Springer-Verlag, 2006--2007.


\bibitem[BGMR]{BGMR}
Bravi, Paolo; Gandini, Jacopo; Maffei, Andrea; Ruzzi, Alessandro,
\emph{Normality and non-normality of group compactifications 
in simple projective spaces}, 
Ann. Inst. Fourier (Grenoble) 61 (2011), no. 6, 2435--2461.


\bibitem[Br]{Brion}
Brion, Michel,
\emph{The total coordinate ring of a wonderful variety},
J. Algebra 313 (2007), 61--99. 


\bibitem[BK]{BK}
Brion, Michel; Kumar, Shrawan
\emph{Frobenius splitting methods in geometry and representation theory},
Progress in Mathematics, 231, Birkh\"auser Boston, Inc., Boston, MA, 2005.


\bibitem[CP]{CP}
Chaput, Pierre-Emmanuel; Perrin, Nicolas,
\emph{On the quantum cohomology of adjoint varieties},
Proc. Lond. Math. Soc. (3) 103 (2011), no. 2, 294--330. 


\bibitem[CFH]{CFH}
Chen, Yifei; Fu, Baohua; Hwang, Jun-Muk,
\emph{Minimal rational curves on complete toric varieties and applications},
Proc. Edinb. Math. Soc. (2) 57 (2014), no. 1, 111--123.




\bibitem[dCP]{dCP}
De Concini, Corrado; Procesi, Claudio, 
\emph{Complete symmetric varieties},
in: Invariant theory (Montecatini, 1982), 1--44, 
Lecture Notes in Math. 996, Springer, Berlin 1983.


\bibitem[D]{Demazure}
Demazure, Michel, 
\emph{Automorphismes et d\'eformations des vari\'et\'es de Borel},
Invent. math. 39 (1977), 179--186.


\bibitem[FH]{FH}
Fu, Baohua; Hwang, Jun-Muk,  
\emph{Classification of non-degenerate projective varieties 
with non-zero prolongation and application to target rigidity}, 
Invent. Math. 189 (2012), no. 2, 457--513.




\bibitem[Ho]{Horrocks}
Horrocks, Geoffrey,
\emph{Fixed point schemes of additive group actions},
Topology 8 (1969), 233--242. 


\bibitem[Hw]{Hwang} 
Hwang, Jun-Muk,  
\emph{Geometry of minimal rational curves on Fano manifolds},   
ICTP Lect. Notes 6 (2001), 335--393.





\bibitem[HM1]{HM1}
Hwang, Jun-Muk; Mok, Ngaiming,
\emph{Deformation rigidity of the rational homogeneous space 
associated to a long simple root},
Ann. Sci. \'Ecole Norm. Sup. (4) 35 (2002), no. 2, 173--184.


\bibitem[HM2]{HM2}
Hwang, Jun-Muk; Mok, Ngaiming,
\emph{Birationality of the tangent map for minimal rational curves},
Asian J. Math. 8 (2004), no. 1, 51--63.


\bibitem[HM3]{HM3}
Hwang, Jun-Muk; Mok, Ngaiming,
\emph{Deformation rigidity of the $20$-dimensional $F_4$-homo\-geneous 
space associated to a short root}, in: 
Algebraic transformation groups and algebraic varieties, 37--58,
Encyclopaedia Math. Sci. 132, Springer, Berlin, 2004. 


\bibitem[HM4]{HM4}
Hwang, Jun-Muk; Mok, Ngaiming,
\emph{Prolongations of infinitesimal linear automorphisms 
of projective varieties and rigidity of rational homogeneous 
spaces of Picard number 1 under K\"ahler deformation},
Invent. Math. 160 (2005), no. 3, 591--645.


\bibitem[Ka]{Kannan}
Kannan, S. Senthamarai, 
\emph{Remarks on the wonderful compactification of semisimple 
algebraic groups}, 
Proc. Indian Acad. Sci. Math. Sci. 109 (1999), no. 3, 241--256.


\bibitem[Ke]{Kebekus}
Kebekus, Stefan, 
\emph{Families of singular rational curves},
J. Algebraic Geom. 11 (2002), no. 2, 245--256. 




\bibitem[Ko]{Kollar}
Koll\'ar, J\'anos, 
\emph{Rational curves on algebraic varieties},
Ergeb. Math. Grenzgeb. (3) 32, Springer-Verlag, 1996.


\bibitem[LM]{LM}
Landsberg, Joseph; Manivel, Laurent,
\emph{On the projective geometry of rational homogeneous varieties},
Comment. Math. Helv. 78 (2003), 65--100.


\bibitem[L]{Luna}
Luna, Domingo,
\emph{Slices \'etales}, 
M\'emoire Soc. Math. France 33 (1973), 81--105.


\bibitem[T]{Timashev}
Timashev, Dmitri,
\emph{Equivariant compactifications of reductive groups},
Sb. Math. 194 (2003), no. 3-4, 589--616. 


\bibitem[V]{Vainsencher}
Vainsencher, Israel,
\emph{Complete collineations and blowing up determinantal ideals},
Math. Ann. 267 (1984), no. 3, 417--432.

\end{thebibliography}
\end{document}